\documentclass[10pt]{amsart}
\usepackage{a4}
\usepackage{amssymb}
\usepackage{amscd}
\newtheorem{thm}{Theorem}

\newtheorem{lem}[thm]{Lemma}
\newtheorem{prop}[thm]{Proposition}
\newtheorem{rem}[thm]{Remark}

\numberwithin{thm}{section}
\numberwithin{equation}{section}

\newcommand{\norm}[1]{\left\Vert#1\right\Vert}
\newcommand{\abs}[1]{\left\vert#1\right\vert}

\newcommand{\n}{\mathbb{N}}

\newcommand{\Hi}{\mathcal{H}}

\newcommand{\K}{\mathcal{K}}
\newcommand{\OTm}{\otimes_{\min}}

\begin{document}
\title{Tsirelson like operator spaces}
\author{Hun Hee Lee}
\address{Department of Mathematical Sciences, Seoul National University San56-1 Shinrim-dong Kwanak-gu Seoul 151-747, Korea}
\email{Lee.hunhee@gmail.com, bbking@amath.kaist.ac.kr}

\begin{abstract}
We construct nontrivial examples of weak-$C_p$ ($1\leq p \leq \infty$) operator spaces
with the local operator space structure very close to $C_p = [R, C]_{\frac{1}{p}}$.
These examples are non-homogeneous Hilbertian operator spaces,
and their constructions are similar to that of 2-convexified Tsirelson's space by W. B. Johnson.
\end{abstract}
\maketitle

\section{Introduction}

Tsirelson's space and its variations have been sources of counterexamples to many questions in the Banach space theory (See \cite{CS} for the details). 
In this paper we focus on the 2-convexified version of Tsirelson's space, which served as an example of nontrivial weak Hilbert space,
a closest object to Hilbert spaces in the sense of type and cotype theory.

Recall that a Banach space $X$ is called a weak Hilbert space (\cite{P0}) if for any $0<\delta <1$ there is a constant $C >0$ with the following property
: for any finite dimensional $F\subseteq X$ we can find $F_1 \subseteq F$ and an onto projection $P :X \rightarrow F_1$ satisfying
$$d_{F_1} := d(F_1, \ell^{\text{dim}F_1}_2) \leq C, \; \text{dim}F_1 \geq \delta \text{dim}F\,\, \text{and} \,\, \norm{P} \leq C,$$
where $d(\cdot,\cdot)$ is the Banach-Mazur distance defined by
$$d(X,Y) = \inf \{ \norm{u}\norm{u^{-1}} | u : X \rightarrow Y,\; \text{isomorphism} \}.$$

As an operator space analogue of weak Hilbert space the author introduced the notion of weak-$H$ spaces
for a (separable and infinite dimensional) perfectly Hilbertian operator space $H$ in \cite{L2}.
A Hilbertian operator space $H$ (i.e. $H$ is isometric to a Hilbert space) is called homogeneous if 
for every $u : H \rightarrow H$ we have $\norm{u}_{cb}  = \norm{u}$ and subquadratic if
for all orthogonal projections $\{P_i\}^n_{i=1}$ in $H$ with $I_H = P_1 + \cdots + P_n$ we have 
$$\norm{x}^2_{B(\ell_2)\otimes_{\min}H} \leq \sum^n_{i=1}\norm{I_{B(\ell_2)}\otimes P_i (x)}^2_{B(\ell_2)\otimes_{\min}H}$$
for any $x\in B(\ell_2)\otimes H$ (See p.82 of \cite{P2}), where $\otimes_{\min}$ is the injective tensor product of operator spaces.
A homogeneous Hilbertian operator space $H$ is called perfectly Hilbertian
if $H$ and $H^*$ is subquadratic. See \cite{L1} for the definition of weak-$H$ space and the related type and cotype notions of operator spaces.

In \cite{L2} it is shown that an operator space $E$ is a weak-$H$ space if and only if
for any $0<\delta <1$ there is a constant $C >0$ with the following property :
for any finite dimensional $F\subseteq E$ we can find $F_1 \subseteq F$ and an onto projection $P :E \rightarrow F_1$ satisfying
$$d^H_{F_1,cb} := d_{cb}(F_1, H_{\text{dim}F_1}) \leq C, \; \text{dim}F_1 \geq \delta \text{dim}F\,\, \text{and} \,\, \norm{P}_{cb} \leq C,$$
where $d_{cb}(\cdot,\cdot)$ is the cb-distance defined by
$$d_{cb}(E,F) = \inf \{ \norm{u}_{cb}\norm{u^{-1}}_{cb} | u : E \rightarrow F,\; \text{isomorphism} \}$$
and $H_n$ is the $n$-dimensional subspace of $H$.
Thus, we can say that weak-$H$ spaces have similar local operator space structure to $H$.

The aim of this paper is to construct nontrivial examples of weak-$H$ spaces (not completely isomorphic to $H$)
for $H = C_p,\, R_p$ and $1\leq p \leq \infty$,
where $C_p = [C, R]_{\frac{1}{p}}$ and $R_p = [R, C]_{\frac{1}{p}}$, interpolation spaces of the column and the row Hilbert space via complex method.
We will follow the approach of W. B. Johnson (and T. Figiel) to construct Hilbertian operator spaces $X_{C_p}$ and $X_{R_p}$ for $1\leq p \leq \infty$, Since $R^*_p = C_{p'}$ for $\frac{1}{p} + \frac{1}{p'} =1$ it is enough to consider $X_{C_p}$ and $X_{R_p}$ for $1\leq p \leq 2$.
Interestingly, the construction for the case $1\leq p<2$ and the case $p=2$ are different although many proofs are overlapping.

In section \ref{sec-prelim} we prepare some back ground materials concerning vector valued Schatten classes and a description of $d^H_{E,cb}$.
In section \ref{Cp} we focus on $1\leq p <2$ case.
In section \ref{subsec-const-basis-Cp} we will construct $X_{C_p}$ (resp. $X_{R_p}$) and investigate the behavior of its canonical basis.
It will be shown that the span of certain block sequences of the canonical basis
is completely isomorphic to $C_p$ (resp. $R_p$) of the same dimension with bounded constants.
In section \ref{subsec-weak-Cp} we will examine that $X_{C_p}$ (resp. $X_{R_p}$) is our desired space.
First, we show that $X_{C_p}$ (resp. $X_{R_p}$) is a weak-$C_p$ (resp. weak-$R_p$).
We prepare additional materials concerning $\pi_{2,H}$-norms, operator space analogue of absolutely 2-summing norm,
and related description of $d^H_{E,cb}$.
Secondly, it will be shown that $X_{C_p}$ (resp. $X_{R_p}$) is not completely isomorphic to $C_p$ (resp. $R_p$)
investigating containment of an isomorphic copy of $c_0$ (the Banach space of all sequences vanishing at infinity) as in chapter 13 of \cite{P1}.

In section \ref{OH} we consider $p=2$ case. Most of arguments from section \ref{Cp} are still available with some exceptions.
In the final remark we construct a non-Hilbertian example of weak-$OH$ space.  

Throughout this paper, we assume that the reader is familiar with basic concepts in Banach spaces (\cite{DJT,P1,TJ}) and operator spaces (\cite{ER,P3}).
In this paper $E$ and $H$ will be reserved for an operator space and a separable, infinite dimensional and perfectly Hilbertian operator space.
Note that $H(I)$ is well-defined for any index set $I$ (\cite{P1.5}). We will simply write $H_n$ when $I = \{1, \cdots, n\}$.
As usual, $B(E, F)$ and $CB(E, F)$ denote the set of all bounded linear maps and all cb-maps from $E$ into $F$, respectively.

$\K$ implies the algebra of compact operators on $\ell_2$ and $\K_0$ is the union of the increasing sequence
$M_1\subseteq \cdots \subseteq M_n \subseteq M_{n+1}\subseteq \cdots$ of matrix algebras. Note that $\K = \overline{\K_{0}}$.

\section{Preiminaries}\label{sec-prelim}
In this section we collect some back ground materials which will be used later.

\subsection{Vector valued Schatten classes}
Let $S_p(\Hi)$ be the Schatten class on a Hilbert space $\Hi$ and $1\leq p \leq \infty$.
In \cite{P2} $S_p(\Hi, E)$, $E$-valued Schatten classes are defined by
$$S_p(\Hi; E) := [S_{\infty}(\Hi)\otimes_{\min} E, S_1(\Hi) \widehat{\otimes} E]_{\frac{1}{p}},$$
where $\otimes_{\min}$ and $\widehat{\otimes}$ refer to injective and projective tensor products of operator spaces.
When $\Hi = \ell_2$ we simply write as $S_p$ and $S_p(E)$.

The above vector valued Schatten classes are useful to describe operator space structure of subspaces of $S_p$.
In particular, the operator space structure of $C_p$, $R_p$ and $\ell_p$ ($c_0$ when $p = \infty$) can be described as follows.
$$\norm{\sum_i x_i \otimes e_i}_{S_p(C_p)} = \norm{\Big(\sum_i x^*_ix_i\Big)^{\frac{1}{2}}}_{S_p},$$
$$\norm{\sum_i x_i \otimes e_i}_{S_p(R_p)} = \norm{\Big(\sum_i x_ix^*_i\Big)^{\frac{1}{2}}}_{S_p}$$
and $$\norm{\sum_i x_i \otimes e_i}_{S_p(\ell_p)} = \Big(\sum_i\norm{x_i}^p_{S_p}\Big)^{\frac{1}{p}}.$$

When we are dealing with $E+_p F$, the sum of operator spaces in the sense of $\ell_p$,
it is more appropriate to use vector valued Schatten classes.
See chapter 2 of \cite{P2} for the definition of $E \oplus_p F$, the direct sum of operator spaces in the sense of $\ell_p$.
Then we define $E+_p F := (E \oplus_p F) / (E\cap F)$. Then we have 
\begin{equation}\label{Sp-sum}
S_p(E \oplus_p F) = S_p(E) \oplus_p S_p(F) \;\, \text{and}\;\, S_p(E +_p F) = S_p(E) +_p S_p(F).
\end{equation}
 
Moreover, we can use vector valued Schatten classes to check complete boundedness.
Indeed, by Lemma 1.7. of \cite{P2} for any cb-map $T:E\rightarrow F$ between operator spaces we have
\begin{equation}\label{Sp-cb-check}
\norm{T}_{cb} = \sup_{n\geq 1} \norm{I_{S^n_p}\otimes T : S^n_p(E) \rightarrow S^n_p(F)}.
\end{equation}

For a subspace $S \subseteq S_p$ we denote $S(E) := \overline{\text{span}}\{S\otimes E\} \subseteq S_p(E)$.
We will focus on the case $S = C_p$, $R_p$ and $\ell_p$ ($c_0$ when $p = \infty$).
By Theorem 1.1. of \cite{P2} we have $$S_p(E) \cong C_p \otimes_h E \otimes_h R_p$$
completely isometrically by the identification $e_{ij} \otimes x \mapsto e_{i1}\otimes x \otimes e_{1j}$,
so that $C_p(E)$ and $R_p(E)$ are 1-completely complemented in $S_p(E)$.

By taking the natural diagonal projection $\ell_p(E)$ is also 1-completely complemented in $S_p(E)$,
and we have $$\ell_p(E) = [c_0 \otimes_{\min} E, \ell_1 \widehat{\otimes} E]_{\frac{1}{p}}.$$

By the above observations for any cb-map $T:E\rightarrow F$ between operator spaces we have
\begin{equation}\label{cb-extension}
\norm{I_S \otimes T : S(E) \rightarrow S(F)}_{cb} = \norm{T}_{cb}
\end{equation}
for $S=C_p$, $R_p$ and $\ell_p$ ($c_0$ when $p = \infty$).

\subsection{A description of $d^H_{E,cb}$ and operator spaces with similar $n$-dimensional structure to $H$}

In this section we present several observation concerning $(2,H)$-summing maps, operator space analogues of 2-summing maps,
and related description of $d^H_{E,cb}$.

For a linear map $T:E\rightarrow F$ between operator maps the $(2,H)$-summing norm $\pi_{2,H}(T)$ is defined by
$$\pi_{2,H}(T) := \sup\Big\{ \frac{(\sum_k \norm{TSe_k}^2)^{\frac{1}{2}}}{\norm{S : H^* \rightarrow E}_{cb}} \Big\}.$$
We need the subquadracity of $H$ to ensure that $\pi_{2,H}(\cdot)$ is a norm.

The factorization norm through $H$ of $T$, $\gamma_{H}(T)$, is defined by $$\gamma_{H}(T) := \inf \{ \norm{T_1}_{cb}\norm{T_2}_{cb}\},$$
where the infimum runs over all possible factorization $$T : E \stackrel{T_1}{\longrightarrow} H^*(I) \stackrel{T_2}{\longrightarrow} F$$
for some index set $I$. Note that the trace dual $\gamma^*_{H}$ of $\gamma_{H}$ can be described as follows (Theorem 6.1. of \cite{P1.5}).
For any finite rank map $T:E\rightarrow F$ we have $$\gamma^*_{H}(T) = \inf \{ \pi_{2,H^*}(T_1)\pi_{2,H}(T^*_2)\},$$
where the infimum runs over all possible factorization $$T : E \stackrel{T_1}{\longrightarrow} \ell^m_2 \stackrel{T_2}{\longrightarrow} F$$
for some $m\in \n$.

We need to consider completely nuclear norm $\nu^o(\cdot)$ (\cite{ER}), which is the trace dual of cb-norm $\norm{\cdot}_{cb}$.
By arguments in p.200-201 of \cite{ER} we have
\begin{equation}\label{trace-nuc}
\abs{\text{tr}(T)} \leq \nu^o(T)
\end{equation}
for all finite rank map $T : E\rightarrow E$.

The following Lemma is an operator space version of the fact that the composition of two 2-summing maps is a nuclear map in some special cases.
\begin{lem}\label{lem-summing-composition}
Let $H$ be a subquadratic, homogeneous and Hilbertian operator space and $E$ be a finite dimensional operator space.
Then for any $u : E\rightarrow E$ we have
$$\nu^o(u) \leq \inf\{\pi_{2,H}(S)\pi^*_{2,H}(T) \},$$ where the infimum runs over all possible factorization
$$u: E \stackrel{S}{\rightarrow} \ell^n_2 \stackrel{T}{\rightarrow}E.$$ 
\end{lem}
\begin{proof}
Let $u : E\rightarrow E$ and consider any factorization $u: E \stackrel{S}{\rightarrow} \ell^n_2 \stackrel{T}{\rightarrow}E.$
Then for any $v : E\rightarrow E$ and any further factorization $T :\ell^n_2 \stackrel{A}{\rightarrow} H^* \stackrel{B}{\rightarrow}E$ we have
\begin{align*}
\abs{\text{tr}(vu)} & = \abs{\text{tr}(vBAS)} = \abs{\text{tr}(SvBA)} \leq \norm{A}_{HS}\norm{SvB}_{HS}\\
& = \norm{A}_{HS} \pi_{2,H}(SvB) \leq \norm{A}_{HS}\norm{B}_{cb}\norm{v}_{cb}\pi_{2,H}(S).
\end{align*}
By taking infimum over all possible $A$, $B$, $S$ and $v$ with $\norm{v}_{cb}\leq 1$ we get the desired result using trace duality.
\end{proof}

Now we present a description of $d^H_{E,cb}$ using $(2,H)$-summing norms.
\begin{lem}\label{lem-description}
Let $H$ be a perfectly Hilbertian operator space.
Then we have $$d^H_{E,cb} = \sup\Big\{\frac{\pi^*_{2,H^*}(u)}{\pi_{2,H}(u^*)} \; | \; u: \ell^m_2 \rightarrow E,\; m\in \n\Big\}.$$ 
\end{lem}
\begin{proof}
We get the result form Theorem 4.3. of \cite{L1} and the fact that
$$\pi_{2,H}(u^*) \leq \ell(u) \leq \pi^*_{2,H^*}(u)$$ for any $u: \ell^m_2 \rightarrow F$.
\end{proof}

We show an operator space version of Remark 13.4. of \cite{P1.5}, which will be useful later.
\begin{prop}\label{prop-complemented}
Let $H$ be a perfectly Hilbertian operator space and $n\in \mathbb{N}$ be fixed.
Suppose that $E$ is an operator space satisfying the following :
there is a constant $C > 0$ such that for any $n$-dimensional subspace $F \subseteq E$ we have $$d^H_{F,cb}\leq C.$$
Then for any $n$-dimensional subspace $F \subseteq E$ we have a projection $P : E\rightarrow E$ onto $F$ with $$\gamma_H(P)\leq C.$$
\end{prop}
\begin{proof}
By combining Lemma \ref{lem-description} and the assumption we get
\begin{equation}\label{summing-inversion}
\pi^*_{2,H^*}(T) \leq C \cdot \pi_{2,H}(T^*)
\end{equation}
for any $T : \ell^n_2 \rightarrow E$.  

Now we fix a $n$-dimensional subspace $F\subseteq E$ and let $i : F \hookrightarrow E$ be the inclusion.
For any $u : F \rightarrow F$ we consider a factorization $$iu : F \stackrel{\alpha}{\rightarrow} \ell^m_2 \stackrel{\beta}{\rightarrow}E.$$
Then by applying Lemma \ref{lem-summing-composition}, \eqref{trace-nuc} and \eqref{summing-inversion} we have
\begin{align*}
\abs{\text{tr}(u)} & \leq \nu^o(u) \leq \pi_{2,H^*}(\alpha)\pi^*_{2,H^*}(\beta|_{\text{ran}\alpha})\\
& \leq C\cdot \pi_{2,H^*}(\alpha)\pi_{2,H}((\beta|_{\text{ran}\alpha})^*)\\
& \leq C\cdot \pi_{2,H^*}(\alpha)\pi_{2,H}(\beta^*).
\end{align*}
By taking infimum over all possible $\alpha$ and $\beta$ we get
$$\abs{\text{tr}(u)} \leq C\cdot \gamma^*_{H}(iu).$$ If we apply Hahn-Banach theorem to the functional $u \mapsto \text{tr}(u)$
we get the desired result.
\end{proof}

\section{The case $1\leq p < 2$}\label{Cp}

\subsection{The construction and basic properties of the canonical basis}\label{subsec-const-basis-Cp}

In this section we will construct Hilbertian operator spaces $X_{C_p}$ and $X_{R_p}$ for $1\leq p <2$,
consequently $X_{C_p}$ and $X_{R_p}$ for $1\leq p \leq \infty, p\neq 2$.
We will mainly focus on $X_{C_p}$ case only since the situation of $X_{R_p}$ is symmetric.

We say that a disjoint collection, $(E_j)^{f(k)}_{j=1}$, of finite subsets of $\mathbb{N}$ is ``allowable" if
$$E_j \subseteq \{k, k+1, \cdots \} \; \text{for all}\; 1\leq j \leq f(k),$$ where $k\in \mathbb{N}$ and $f(k) = (4k^3)^k$.
This specific choice of $f$ will be clarified later in section \ref{sec-example-weakH}.
For a finite subset $E \subseteq \mathbb{N}$ and $$x = \sum_{i \geq 1}x_i\otimes t_i \in \K_0 \otimes c_{00},$$
where $t_i$ is the $i$-th unit vector in $c_{00}$ (finitely supported sequences of complex numbers),
we denote $$Ex = \sum_{i \in E}x_i\otimes t_i.$$

Let $1\leq p < 2$ and $0<\theta <1$ be fixed.
We will define a sequence of norms on $\K_0 \otimes c_{00}$ to construct $X_{C_p}$ (resp. $X_{R_p}$).
For $x \in \K_0 \otimes c_{00}$ we define $$\norm{x}_0 = \norm{x}_{\K \otimes_{\min} (R_p +_p C_p)}.$$
Then $X_0$, the completion of $(c_{00},\norm{\cdot}_0)$, is nothing but the homogeneous Hilbertian operator space $R_p +_p C_p$,
and clearly $\norm{\cdot}_0$ satisfies Ruan's axioms. Now we define $(\norm{\cdot}_n)_{n\geq 0}$ recursively.
Suppose that $\norm{\cdot}_n$ is already defined and satisfies Ruan's axioms (\cite{ER, P3}).
Then $X_n$, the completion of $(c_{00},\norm{\cdot}_n)$ is an operator space, and we define
\begin{align*}
\norm{x}_{n+1} & = \max\Big\{\norm{x}_n, \theta \sup \norm{\sum^{f(k)}_{j=1}e_{j1} \otimes E_j x}_{\K \otimes_{\min} C_p(X_n)}\Big\}\\
& (\text{resp.}\; \max\Big\{\norm{x}_n, \theta \sup \norm{\sum^{f(k)}_{j=1}e_{1j} \otimes E_j x}_{\K \otimes_{\min} R_p(X_n)}\Big\}),
\end{align*}
where the inner supremum runs over all ``allowable" sequence $\{E_j\}^{f(k)}_{j=1} \subseteq \mathbb{N}$.
Then $\norm{\cdot}_{n+1}$ satisfies Ruan's axioms, so that $X_{n+1}$, the completion of $(c_{00},\norm{\cdot}_{n+1})$ is an operator space.
Actually, $X_{n+1}$ is a subspace of $X_n \oplus_{\infty} \ell_{\infty}(I; \{C_p(X_n)\})$ spanned by elements of the form,
$(x, (\theta E_jx)_{(E_j) \in I})$, where $I$ is the collection of all allowable sequences,
so that $X_{n+1}$ inherits the operator space structure from $X_n \oplus_{\infty} \ell_{\infty}(I; \{C_p(X_n)\})$
(the case for $R_p$ is similar).

\begin{rem}{\rm
When we write $e_{j1}\otimes E_j x \in \K \otimes_{\min} C_p(X_n)$ one should note that $e_{j1} \in C_p$ and $E_j x \in \K \otimes X_n$,
which is twisted in order.
}
\end{rem}

\begin{prop}\label{prop-Hilbertian-Cp}
For any $x \in \K_0 \otimes c_{00}$, $(\norm{x}_n)_{n\geq 0}$ is increasing, and we have
$$\norm{x}_{\K \otimes_{\min}(R_p +_p C_p)} \leq \norm{x}_n \leq \norm{x}_{\K \otimes_{\min} C_p}$$
for all $n \geq 0$.
\end{prop}
\begin{proof}
The left inequality is clear. For the right inequality we use induction on $n$. When $n=0$, it is trivial.
Suppose we have the right inequality for $n$ and for all $x \in \K_0 \otimes c_{00}$, equivalently,
the formal identity $C_p \rightarrow X_n$ is completely contractive. Then we have
\begin{align*}
\theta\norm{\sum^{f(k)}_{j=1}e_{j1} \otimes E_j x}_{S_p(C_p(X_n))}
& = \theta\norm{\sum^{f(k)}_{j=1}\sum_{i \in E_j}x_i\otimes e_i \otimes e_{j1}}_{S_p(C_p(X_n))}\\
& \leq \theta\norm{\sum^{f(k)}_{j=1}\sum_{i \in E_j}x_i\otimes e_{i1} \otimes e_{j1}}_{S_p(C_p(C_p))}\\
& = \theta\norm{\Big(\sum^{f(k)}_{j=1}\sum_{i \in E_j}x^*_i x_i\Big)^{\frac{1}{2}}}\\
& \leq \theta\norm{x}_{S_p(C_p)} < \norm{x}_{S_p(C_p)}.
\end{align*}

Thus, we have that $$x \mapsto \theta \sum^{f(k)}_{j=1}e_{j1} \otimes E_j x, \; C_p \rightarrow C_p(X_n)$$ is a complete contraction, and so is
$$x \mapsto x\oplus \theta \sum^{f(k)}_{j=1}e_{j1} \otimes E_j x, \; C_p \rightarrow X_n\oplus_{\infty}C_p(X_n)$$ by the assumption,
which leads us to the right inequality for $n+1$ and for all $x \in \K_0 \otimes c_{00}$.
\end{proof}

Now we can consider $\norm{x} = \lim_{n\rightarrow \infty}\norm{x}_n$ for all $x \in \K_0 \otimes c_{00}$,
and clearly $\norm{\cdot}$ satisfies Ruan's axioms, so that $X_{C_p}$ (resp. $X_{R_p}$),
the completion of $(c_{00},\norm{\cdot})$ is an operator space.
Actually, $X_{C_p}$ (resp. $X_{R_p}$) is a subspace of $\ell_{\infty}(X_n)$ spanned by elements of the form, $(x, x, \cdots, x)$,
so that $X_{C_p}$ (resp. $X_{R_p}$) inherits the operator space structure from $\ell_{\infty}(X_n)$.
Moreover, $X_{C_p}$ (resp. $X_{R_p}$) is Hilbertian by Proposition \ref{prop-Hilbertian-Cp}.

We have a slight different form of $\norm{\cdot}_n$ which will be useful later.

\begin{prop}\label{prop-another-def-Cp}
For any $x \in \K_0 \otimes c_{00}$ and any $n\geq 0$ we have
$$\norm{x}_{n+1} = \max\Big\{\norm{x}_0, \theta \sup \norm{\sum^{f(k)}_{j=1}e_{j1} \otimes E_j x}_{\K \otimes_{\min} C_p(X_n)}\Big\},$$
where the inner supremum runs over all ``allowable" sequence $\{E_j\}^{f(k)}_{j=1} \subseteq \mathbb{N}$.
\end{prop}
\begin{proof}
Suppose we have $$\norm{x}_{n+1} > \theta \sup \Big\{\norm{\sum^{f(k)}_{j=1}e_{j1} \otimes E_j x}_{\K \otimes_{\min} C_p(X_n)}\Big\}$$
for an $x \in \K_0 \otimes c_{00}$. Then by the definition of $\norm{\cdot}_{n+1}$ we have $\norm{x}_{n+1} = \norm{x}_n$.
Since the formal identity $i_n : X_{n} \rightarrow X_{n-1}$ is clearly a complete contraction we get another complete contraction
$$I_{C_p}\otimes i_n : C_p(X_n) \rightarrow C_p(X_{n-1})$$ by \eqref{cb-extension}. Thus, it follows that
\begin{align*}
\norm{x}_{n} & > \theta \sup \Big\{\norm{\sum^{f(k)}_{j=1}e_{j1} \otimes E_j x}_{\K \otimes_{\min} C_p(X_n)}\Big\}\\
& \geq \theta \sup \Big\{\norm{\sum^{f(k)}_{j=1}e_{j1} \otimes E_j x}_{\K \otimes_{\min} C_p(X_{n-1})}\Big\}
\end{align*}
and hence $\norm{x}_{n} = \norm{x}_{n-1}$. If we repeat this process, then we get $\norm{x}_{n+1} = \norm{x}_0$.
\end{proof}

We say that a basis $\{f_i\}_{i \geq 1}$ of an operator space $E$ is $C$-completely unconditional if
$$\norm{\sum_{i \geq 1}a_i x_i\otimes f_i}_{\K \otimes_{\min}E} \leq C \norm{\sum_{i \geq 1}x_i\otimes f_i}_{\K \otimes_{\min}E}$$
for any finitely supported sequence of $\{x_i\}_{i \geq 1}\subseteq \K_0$ and any sequence of scalars
$(a_i)_{i \geq 1}$ with $\abs{a_i} \leq 1$ for all $i \geq 1$.

\begin{prop}\label{prop-basis-Cp}
The canonical basis $\{t_i\}_{i\geq 1}$ is a normalized 1-completely unconditional basis for $X_{C_p}$.
\end{prop}
\begin{proof}
We will use induction on $n$, to show that $\{t_i\}_{i\geq 1}$ is a normalized 1-completely unconditional basis for $X_n$.
First, we fix a sequence $(a_i)_{i\geq 1}$ with $\abs{a_i} \leq 1$ for all $i \geq 1$.

When $n=0$, for any $\sum_{i \geq 1}x_i\otimes e_i \in \K_0 \otimes c_{00}$ we have
\begin{align*}
\lefteqn{\norm{\sum_{i\geq 1}a_i x_i\otimes t_i}^p_{S_p(R_p+_p C_p)}}\\
& = \inf_{x_i = y_i + z_i}\Big\{\norm{\Big(\sum_{i\geq 1}\abs{a_i}^2 y_iy^*_i\Big)^{\frac{1}{2}}}^p_{S_p} +
\norm{\Big(\sum_{i\geq 1}\abs{a_i}^2 z^*_iz_i\Big)^{\frac{1}{2}}}^p_{S_p}\Big\}\\
& \leq \inf_{x_i = y_i + z_i}\Big\{\norm{\Big(\sum_{i\geq 1}y_iy^*_i\Big)^{\frac{1}{2}}}^p_{S_p}+
\norm{\Big(\sum_{i\geq 1}z^*_iz_i\Big)^{\frac{1}{2}}}^p_{S_p}\Big\}\\
& = \norm{\sum_{i\geq 1} x_i\otimes t_i}^p_{S_p(R_p +_p C_p)},
\end{align*}
since $\abs{a_i}^2 y_iy^*_i \leq y_iy^*_i$ and $\abs{a_i}^2 z^*_iz_i \leq z^*_iz_i$.

Suppose we have the result for $n$, which is equivalent to $$X_n \rightarrow X_n,\; t_i \mapsto a_it_i$$ is a complete contraction.
Then by \eqref{cb-extension} we have another complete contraction
$$C_p(X_n) \rightarrow C_p(X_n),\; e_{j1}\otimes t_i \mapsto e_{j1}\otimes a_it_i.$$
Thus, for any ``allowable" sequence $\{E_j\}^{f(k)}_{j=1} \subseteq \mathbb{N}$ we have
\begin{align*}
\lefteqn{\theta \norm{\sum^{f(k)}_{j=1}e_{j1} \otimes E_j\Big(\sum_i a_ix_i\otimes t_i\Big)}_{\K \otimes_{\min} (C_p(X_n))}}\\
& = \theta \norm{\sum^{f(k)}_{j=1} \sum_{i\in E_j}x_i \otimes e_{j1} \otimes a_it_i}_{\K \otimes_{\min} (C_p(X_n))}\\
& \leq \theta \norm{\sum^{f(k)}_{j=1} \sum_{i\in E_j}x_i \otimes e_{j1} \otimes t_i}_{\K \otimes_{\min} (C_p(X_n))}\\
& = \theta \norm{\sum^{f(k)}_{j=1}e_{j1} \otimes E_j\Big(\sum_i x_i\otimes t_i\Big)}_{\K \otimes_{\min} (C_p(X_n))}
\leq \norm{\sum_i x_i \otimes t_i}_{n+1},
\end{align*}
which implies the result for $n+1$.
\end{proof}

Now we investigate the operator space structure of the subspace spanned by certain normalized and disjoint block sequences of
$\{t_i\}_{i\geq 1} \subseteq X_{C_p}$. They are $\theta$-completely isomorphic to $C_p$ with the same dimension.

\begin{prop}\label{prop-block-oss}
Let $(y_j)^{f(k)}_{j=1}$ be a disjoint and normalized block sequences of $\{t_i\}_{i\geq 1}\subseteq X_{C_p}$
with $\text{\rm supp}y_j \subseteq \{k, k+1, \cdots\}$ for $1\leq j \leq f(k)$. Then we have
$$\theta \norm{\sum^{f(k)}_{j=1} b_j \otimes e_{j1}}_{S_p(C_p)} \leq \norm{\sum^{f(k)}_{j=1} b_j \otimes y_j}_{S_p(X_{C_p})}
\leq \norm{\sum^{f(k)}_{j=1} b_j \otimes e_{j1}}_{S_p(C_p)}$$
for any $(b_j)^{f(k)}_{j=1} \subseteq S_p$.
\end{prop}
\begin{proof}
For the left inequality we set $E_j = \text{supp}y_j$ and $n_j = \min \{\text{supp}y_j\}$. Since
$$X_{C_p} \rightarrow C_p(X_n), \; x\mapsto \theta \sum^{f(k)}_{j=1}e_{j1} \otimes E_j x$$ is a complete contraction we have
\begin{align*}
\norm{\sum^{f(k)}_{i=1} b_i \otimes y_i}_{S_p(X_{C_p})}
& \geq \theta \norm{\sum^{f(k)}_{j=1}e_{j1}\otimes E_j\Big(\sum^{f(k)}_{i=1} b_i\otimes y_i\Big)}_{S_p(C_p(X_n))}\\
& = \theta \norm{\sum^{f(k)}_{j=1}b_j \otimes e_{j1} \otimes y_j}_{S_p(C_p(X_n))}\\
& \geq \theta \norm{\sum^{f(k)}_{j=1}b_j \otimes e_{j1} \otimes t_{n_j}}_{S_p(C_p(X_n))}\\
& \geq \theta \norm{\sum^{f(k)}_{j=1}b_j \otimes e_{j1}}_{S_p(C_p)}
\end{align*}
by Lemma \ref{lem-block-lower-bound} and \ref{lem-emerging} below for any $(b_i)^{f(k)}_{j=1} \subseteq S_p$.

For the right inequality we will show the following.
For any disjoint and normalized sequence $(y_j)^{f(k)}_{j=1} \subseteq X_{C_p}$ we have
\begin{equation}\label{stronger}
\norm{\sum^{f(k)}_{j=1} b_j \otimes y_j}_{S_p(X_n)} \leq \norm{\sum^{f(k)}_{j=1} b_j \otimes e_{j1}}_{S_p(C_p)}
\end{equation}
for all $(b_j)^{f(k)}_{j=1} \subseteq S_p$.

Let us use induction on $n$.
When $n=0$ we are done since $$CB(C_p, R_p+_p C_p) = B(C_p, R_p+_p C_p)$$ isometrically
and $(e_{j1})_{j\geq 1}$ and $(y_j)^{f(k)}_{j = 1}$ are orthonormal.
Suppose we have \eqref{stronger} for $n$. Then for any ``allowable" sequence $\{E_j\}^{f(l)}_{j=1} \subseteq \mathbb{N}$ we have
\begin{align*}
\lefteqn{\theta \norm{\sum^{f(l)}_{j=1}e_{j1}\otimes E_j\Big(\sum^{f(k)}_{i=1} b_i\otimes y_i\Big)}_{S_p(C_p(X_n))}}\\
& = \norm{\sum^{f(l)}_{j=1}\sum^{f(k)}_{i=1}b_i\otimes e_{j1}\otimes \theta E_j y_i }_{S_p(C_p(X_n))}\\
& \leq \norm{\sum^{f(l)}_{j=1}\sum^{f(k)}_{i=1}b_i\otimes e_{j1}\otimes \norm{\theta E_j y_i}e_{1, ij} }_{S_p(C_p(C_p))}\\
& = \norm{\Big(\sum^{f(l)}_{j=1}\sum^{f(k)}_{i=1}b^*_i b_i \norm{\theta E_j y_i}^2\Big)^{\frac{1}{2}}}_{S_p}
\leq \norm{\Big(\sum^{f(k)}_{i=1}b^*_i b_i \Big)^{\frac{1}{2}}}_{S_p}.
\end{align*}
The last inequality comes from $$\sum^{f(l)}_{j=1} \norm{\theta E_j y_i}^2 \leq \theta^2 \norm{y_i}^2 \leq 1.$$
This conclude \eqref{stronger} for $n+1$ as before.
 
\end{proof}

\begin{lem}\label{lem-block-lower-bound}
Let $(y_j)^{f(k)}_{j=1}$, $(E_j)^{f(k)}_{j=1}$ and $(n_j)^{f(k)}_{j=1}$ be the same as in Proposition \ref{prop-block-oss}. Then we have
\begin{equation}\label{lower}
\norm{\sum^{f(k)}_{j=1} b_j \otimes t_{n_j}}_{S_p(X_n)} \leq \norm{\sum^{f(k)}_{j=1} b_j \otimes y_j}_{S_p(C_p)}
\end{equation}
for all $(b_j)^{f(k)}_{j=1} \subseteq S_p$.
\end{lem}
\begin{proof}
We use induction on $n$. When $n=0$ we are done since $R_p+_p C_p$ is a homogeneous Hilbertian operator space
and $(t_{n_j})^{f(k)}_{j= 1}$ and $(y_j)^{f(k)}_{j = 1}$ are orthonormal.
Suppose we have \eqref{lower} for $n$, and consider any fixed ``allowable" sequence $\{F_j\}^{f(k)}_{j=1} \subseteq \mathbb{N}$.
If we set $$x = \sum^{f(k)}_{j=1} b_j \otimes t_{n_j},\; y = \sum^{f(k)}_{j=1} b_j \otimes y_j$$
and $$G_j = \{ n_i : n_i \in F_j\}\; \text{for}\; 1\leq j \leq f(k),$$ then $$G_jy = \sum_{n_i \in F_j}b_i \otimes t_{n_i}$$
and $\{G_j\}^{f(k)}_{j=1}$ is an ``allowable" sequence. Consequently, we have
\begin{align*}
\theta \norm{\sum^{f(k)}_{j=1} e_{j1} \otimes F_j x}_{S_p(C_p(X_n))}
& = \theta \norm{\sum^{f(k)}_{j=1} \sum_{n_i \in F_j}b_i \otimes e_{j1} \otimes t_{n_i}}_{S_p(C_p(X_n))}\\
& = \theta \norm{\sum^{f(k)}_{j=1} e_{j1} \otimes G_jy }_{S_p(C_p(X_n))}\\
& \leq \norm{y}_{S_p(X_{n+1})}.
\end{align*}
The last line is obtained by the complete contraction
$$X_{n+1} \rightarrow C_p(X_n), \; z\mapsto \theta \sum^{f(k)}_{j=1}e_{j1} \otimes G_j z.$$
This conclude \eqref{lower} for $n+1$ as before.
\end{proof}

\begin{lem}\label{lem-emerging}
We have $$\norm{\sum_{j\geq 1} b_j \otimes e_{j1} \otimes t_j}_{S_p(C_p(X_n))} \geq \norm{\sum_{j \geq 1} b_j \otimes e_{j1}}_{S_p(C_p)}$$
for all finitely supported $(b_j)_{j\geq 1} \subseteq S_p$.
\end{lem}
\begin{proof}
Consider a finitely supported $(b_j)_{j\geq 1} \subseteq S_p$. Then we have
$$\norm{\sum_{j\geq 1} b_j \otimes e_{j1} \otimes t_j}_{S_p(C_p(X_n))}
\geq \norm{\sum_{j\geq 1} b_j \otimes e_{j1} \otimes e_j}_{S_p(C_p(R_p +_p C_p))}.$$
Note that $C_p(R_p +_p C_p) \cong C_p(R_p) +_p C_p(C_p) \cong S_p +_p C_p$ completely isometrically.
Since the formal identities $\ell_1 \rightarrow C_1$ and $\ell_2 \rightarrow C_2$ are complete contractions so is
$\ell_p \rightarrow C_p$ by complex interpolation. Since $\overline{\text{span}}\{e_{j1}\otimes e_j\}_{j\geq 1}$
correspond to $\ell_p$ and $C_p$ in $S_p$ and $C_p(C_p)$, respectively, we have
$$\overline{\text{span}}\{e_{j1}\otimes e_j\}_{j\geq 1} (\subseteq C_p(R_p +_p C_p))\cong C_p$$ completely isometrically.
Thus, we have $$\norm{\sum_{j\geq 1} b_j \otimes e_{j1} \otimes t_j}_{S_p(C_p(X_n))}
= \norm{\sum_{j\geq 1} b_j \otimes e_{j1}}_{S_p(C_p)}.$$
\end{proof}

\subsection{$X_{C_p}$ is a nontrivial weak-$C_p$ space}\label{subsec-weak-Cp}

\begin{prop}\label{prop-Yn}
For $n\in \mathbb{N}$ we consider $$Y_n = \overline{\text{\rm span}}\{t_i\}_{i\geq n+1} \subseteq X_{C_p}.$$
Then for any $E\subseteq Y_n$  with $\text{dim}E = n$, we have
$$d_{cb}(E, C^n_p) \leq 3\theta^{-1}.$$
\end{prop}
\begin{proof}
By Proposition V.6 of \cite{CS} there is a linear map $V : E \rightarrow Y_n$ such that $V(E) \subseteq \overline{\text{\rm span}}\{y_i\}^{f(n)}_{i=1}$,
where $y_i$'s are disjoint elements in $Y_n$ and $$\norm{Vf - f} \leq \frac{1}{2n}\norm{f}$$ for all $f\in E$.
Now we consider the Auerbach basis $(x_i , x^*_i)^{n}_{i=1}$ of $E$. Then we have
$$\sum^{3n}_{i=1}\norm{x^*_i}\norm{x_i -Vx_i} \leq \frac{1}{2n}\sum^{n}_{i=1}\norm{x_i} \leq \frac{1}{2},$$
which implies $$d_{cb}(E,V(E)) \leq 3$$ by the perturbation lemma (Lemma 2.13.2 of \cite{P3}).
On the other hand we have $$d_{cb}(\overline{\text{\rm span}}\{y_i\}^{f(n)}_{i=1}, C^{f(n)}_p) \leq \theta^{-1}$$
by Proposition \ref{prop-block-oss}. Consequently, by combining these we get our desired result.
\end{proof}

\begin{rem}{\rm
The employment of Proposition V.6 of \cite{CS} in the proof of Proposition \ref{prop-Yn} is the reason why we have chosen $f(k) = (4k^3)^k$.}
\end{rem}

Actually we can show that every $n$-dimensional subspace of $Y_n$ in Proposition \ref{prop-Yn} is completely complemented with bounded constants,
so that we are ready to prove one of our main results.

\begin{thm}\label{thm-main1-Cp}
$X_{C_p}$ is a weak-$C_p$ space.
\end{thm}
\begin{proof}
Let $Y_n$ be the same as in Proposition \ref{prop-Yn}. Consider any $E\subseteq X_{C_p}$ with $$\text{dim}E = 2n \; \text{or}\; 2n+1.$$
Then we have $\text{dim}(E\cap Y_n) \geq n$, so that there is $F \subseteq E$ such that $\text{dim}F = n$ and $F \subseteq Y_n$.
Then by Proposition \ref{prop-Yn} we have $$d^H_{F,cb} \leq 3\theta^{-1}.$$
Moreover, $F$ is $3\theta^{-1}$-completely complemented in $Y_n$ by Proposition \ref{prop-complemented}.
Since $Y_n$ itself is $1$-completely complemented in $X_{C_p}$ so is $F$, which implies $X_{C_p}$ is a weak-$C_p$ space.
\end{proof}

All we have to do now is to show that $X_{C_p}$ is not completely isomorphic to $C_p$.
Before that we need to prepare the following lemmas which are analogues of Lemma 13.6 and 13.7 in \cite{P1}.

\begin{lem}\label{lem1-Cp}
Let $N \in \mathbb{N}$ be fixed. Then, for any $n\geq 0$ and any $y,z\in C_p \otimes X_{C_p}$ with
$$\text{\rm supp}y \subseteq \{1, 2, \cdots, N\}\; \text{and}\;\, \text{\rm supp}z \subseteq \{N+1, N+2, \cdots \}$$ we have
$$\norm{y+z}_{C_p \otimes_{\min}X_{n+1}} \leq
\max\{\norm{y}_{C_p \otimes_{\min}X_{n+1}} + \alpha \norm{z}_{C_p \otimes_{\min}X_n},\; \norm{z}_{C_p \otimes_{\min}X_{n+1}}\},$$
where $\alpha = \max\{1, \theta f(N)\}$.
\end{lem}
\begin{proof}
Let us fix $n\geq 0$ and $y, z\in C_p \otimes X_{C_p}$ with
$$\text{supp}y \subseteq \{1, 2, \cdots, N\}\; \text{and}\;\, \text{supp}z \subseteq \{N+1, N+2, \cdots \}.$$
Let $T^y$, $T^z$ and $T^{y+z}$ are linear maps from $R_p$ into $X_{C_p}$ associated with $y$, $z$ and $y+z$, respectively.

Then we have $$\norm{y+z}_{C_p \otimes_{\min}X_{n+1}} = \norm{T^{y+z} : R_p \rightarrow X_{n+1}}_{cb}.$$
For any $x = \sum_i x_i \otimes e_{1i} \in \K \otimes R_p$ we consider $\norm{T^{y+z}(x)}_{n+1}$.
If $$\norm{T^{y+z}(x)}_{n+1} = \norm{T^{y+z}(x)}_n,$$ then we have
\begin{align*}
\norm{T^{y+z}(x)}_{n+1} & \leq \norm{T^{y}(x)}_n + \norm{T^{z}(x)}_n\\
& \leq \Big(\norm{T^{y} : R_p \rightarrow X_n}_{cb} + \norm{T^{z} : R_p \rightarrow X_n}_{cb}\Big) \norm{x}_{\K \otimes R_p}\\
& \leq \Big(\norm{y}_{C_p \otimes_{\min}X_{n+1}} + \norm{z}_{C_p \otimes_{\min}X_n}\Big) \norm{x}_{\K \otimes R_p}.
\end{align*}

If not, we consider any ``allowable" sequence $\{E_j\}^{f(k)}_{j=1} \subseteq \mathbb{N}$.
When $k > N$, we have
\begin{align*}
\theta \norm{\sum^{f(k)}_{j=1} e_{j1} \otimes E_j[T^{y+z}(x)]}_{\K \otimes_{\min} C_p(X_n)}
& = \theta \norm{\sum^{f(k)}_{j=1} e_{j1} \otimes E_j[T^{z}(x)]}_{\K \otimes_{\min} C_p(X_n)}\\
& \leq \norm{T^{z}(x)}_{n+1} \leq \norm{z}_{C_p \otimes_{\min}X_{n+1}}\norm{x}_{\K \otimes R_p}.
\end{align*}
Otherwise, we have
\begin{align*}
\lefteqn{\theta \norm{\sum^{f(k)}_{j=1} e_{j1} \otimes E_j[T^{y+z}(x)]}_{\K \otimes_{\min} C_p(X_n)}}\\
& = \theta \norm{\sum^{f(k)}_{j=1} e_{j1} \otimes E_j[T^{y}(x)]}_{\K \otimes_{\min} C_p(X_n)}
+ \theta \norm{\sum^{f(k)}_{j=1} e_{j1} \otimes E_j[T^{z}(x)]}_{\K \otimes_{\min} C_p(X_n)}\\
& \leq \norm{T^y(x)}_{n+1} + \theta \sum^{f(k)}_{j=1}\norm{E_j[T^{z}(x)]}_n\\
& \leq \Big(\norm{y}_{C_p \otimes_{\min}X_{n+1}}+ \theta N \norm{z}_{C_p \otimes_{\min}X_n}\Big)\norm{x}_{\K \otimes R_p}.
\end{align*}
Combining the above results we get our desired estimate. 

\end{proof}

\begin{lem}\label{lem2-Cp}
For any $n\geq 0$ and any $x\in \K \otimes_{\min} X_{C_p}$ we have
\begin{equation}\label{decomposing}
\norm{x}_{\K \otimes_{\min} X_{C_p}} \leq \norm{(x,\theta^n x)}_{\K \otimes_{\min} (X_n \oplus_p C_p)}.
\end{equation}

\end{lem}
\begin{proof}
We will use induction on $n$ to show (\ref{decomposing}) for all $x\in \K_0 \otimes c_{00}$.

When $n=0$, it is trivial. Suppose we have (\ref{decomposing}) for all $x\in \K_0 \otimes c_{00}$ and for a fixed $n \geq 0$,
which is equivalent to the fact that the map $$(x, \theta^n x) \mapsto x, \; F \rightarrow X_{C_p}$$ is a complete contraction,
where $F = \overline{\text{span}}\{(y,\theta^n y) : y \in c_{00}\} \subseteq X_n \oplus_p C_p$.
Then, for any ``allowable" sequence $\{E_j\}^{f(k)}_{j=1} \subseteq \mathbb{N}$ we have another complete contraction
$$e_{j1} \otimes (x, \theta^n x) \mapsto e_{j1} \otimes x, \,\, C_p(F) \rightarrow C_p(X_{C_p}).$$
Thus, we have
\begin{align*}
\lefteqn{\theta\norm{\sum^{f(k)}_{j=1}e_{j1} \otimes E_j x}_{\K \otimes_{\min} C_p(X_{C_p})}} \\
& \leq \norm{\sum^{f(k)}_{j=1}e_{j1} \otimes (\theta E_j x, \theta^{n+1}E_j x)}_{\K \otimes_{\min} C_p(X_n \oplus_p C_p)} \\
&=\norm{\Big(\sum^{f(k)}_{j=1}e_{j1}\otimes\theta E_j x,
\sum^{f(k)}_{j=1}e_{j1}\otimes\theta^{n+1}E_j x\Big)}_{\K \otimes_{\min} [C_p(X_n) \oplus_p C_p(C_p)]}\\
& \leq \norm{(x,\theta^{n+1} x)}_{\K \otimes_{\min} (X_{n+1} \oplus_p C_p)}
\end{align*}
The last line follows from the fact that the maps
$$x\mapsto \theta \sum^{f(k)}_{j=1}e_{j1}\otimes E_j x, \; X_{n+1} \rightarrow C_p(X_n)$$ and
$$x\mapsto \sum^{f(k)}_{j=1}e_{j1}\otimes E_j x, \; C_p \rightarrow C_p(C_p)$$ are complete contractions.

Since $$\norm{x}_{m+1} = \max\Big\{\norm{x}_0, \theta \sup \norm{\sum^{f(k)}_{j=1}e_{j1}\otimes E_j x}_{\K \otimes_{\min} C_p(X_m)}\Big\}$$
for all $m \geq 0$ by Proposition \ref{prop-another-def-Cp}, we have
$$\norm{x} \leq \max\Big\{\norm{x}_0, \theta \sup \norm{\sum^{f(k)}_{j=1}e_{j1}\otimes E_j x}_{\K \otimes_{\min} C_p(X_{C_p})}\Big\},$$
which leads us to our desired conclusion.
\end{proof}

The following proposition is the crucial point to explain why we cannot have complete isomorphism between $C_p$ and $X_{C_p}$.

\begin{prop}\label{prop-c0-Cp}
Let $n\geq 0$. Then, $C_p \otimes_{\min} E$ contains an isomorphic copy of $c_0$ for any infinite dimensional subspace $E \subseteq X_n$.
\end{prop}
\begin{proof} We will use induction on $n$. Consider $n=0$. Since
$$C_p \otimes_{\min}(R_p +_p C_p) \hookrightarrow B(\ell_2),\; e_{j1}\otimes e_j\mapsto e_{jj}$$
$(e_{j1}\otimes e_j)_{j \geq 1}$ is a basic sequence in $C_p \otimes_{\min}(R_p +_p C_p)$ equivalent to the canonical basis of $c_0$.
Since $R_p +_p C_p$ is homogeneous, every infinite dimensional subspace $E \subseteq R_p +_p C_p$ is completely isometric to $R_p +_p C_p$ itself.
Thus, we get the desired result for $n=0$.

Now suppose that $C_p \otimes_{\min} E$ contains an isomorphic copy of $c_0$ for any infinite dimensional subspace $E \subseteq X_n$.
Let $F \subseteq X_{n+1}$ be infinite dimensional, and $\epsilon>0$ be arbitrarily given.
We claim that there is an infinite dimensional subspace $F' \subseteq F$ such that $C_p \OTm F'$ and $C_p \OTm X_n|_{F'}$ are isomorphic
or we can choose a sequence $(x_i)_{i\geq 1} \subseteq C_p \OTm F$ satisfying $$\norm{x_i}_{C_p \OTm X_{n+1}} = 1$$ for all $i \geq 1$ and
$$\norm{\sum^M_{i=1}x_i}_{C_p \OTm X_{n+1}} \leq 1+ \epsilon$$ for all $M \geq 1$.
Both cases imply $C_p \OTm F$ contains an isomorphic copy of $c_0$, and consequently we get our desired induction result.

For the claim we start with a norm 1 vector $x_1 = e_{11}\otimes x\in C_p \OTm F$.
Suppose we have disjoint and finitely supported $x_1, \cdots, x_m \in C_p \OTm F$ with
$$\norm{x_i}_{C_p \OTm X_{n+1}} = 1$$ for all $1\leq i \leq m$ and
$$\norm{\sum^{m}_{i=1}x_i}_{C_p \OTm X_{n+1}} \leq 1 + \epsilon\sum^{m-1}_{i=1}\frac{1}{2^i}.$$
Let $N$ be a natural number such that $N \geq \text{supp}x_i$ for all $1\leq i \leq m$ and 
$$Y_N = \overline{\text{span}}\{e_i\}_{i\geq N+1} \subseteq X_{n+1}.$$
If $C_p \OTm F\cap Y_N$ and $C_p \OTm X_n|_{F\cap Y_N}$ are isomorphic, then it is done by the induction hypothesis.
Suppose $C_p \OTm F\cap Y_N$ and $C_p \OTm X_n|_{F\cap Y_N}$ are not isomorphic.
Then there is a finitely supported $$x_{m+1} \in C_p \OTm F\cap Y_N\; \text{with}\; \text{supp}x_{m+1} \subseteq \{N+1, N+2, \cdots \}$$
satisfying $$\norm{x_{m+1}}_{C_p \OTm X_{n+1}} = 1\; \text{and}\; \norm{x_{m+1}}_{C_p \OTm X_n} < \frac{\epsilon}{2^m\alpha},$$
where $\alpha = \max\{1, \theta f(N)\}$. Then by Lemma \ref{lem1-Cp} we have
$$\norm{\sum^{m+1}_{i=1}x_i}_{C_p \OTm X_{n+1}} \leq 1 + \epsilon\sum^{m}_{i=1}\frac{1}{2^i}.$$
By repeating this process we get our claim.
\end{proof}

\begin{thm}\label{thm-main2-Cp}
$X_{C_p}$ is not completely isomorphic to $C_p$.
\end{thm}
\begin{proof}
Suppose that $X_{C_p}$ is $C$-completely isomorphic to $C_p$ for some $C>0$.
Then $X_{C_p}$ is $(C+\epsilon)$-homogeneous for any $\epsilon >0$.
Thus, by repeating the proof of Proposition 10.1 in \cite{P3} to $X_{C_p}$ and $C_p$ we get
$$\norm{I}_{cb} = \norm{I}_{cb}\norm{I^{-1}}_{cb} \leq (C+\epsilon)^2 d_{cb}(X_{C_p}, C_p) \leq (C+\epsilon)^3,$$
where $I : X_{C_p} \rightarrow C_p$ is the formal identity. By Lemma \ref{lem2-Cp} and \eqref{Sp-sum} we have
\begin{align*}
(C+\epsilon)^{-3p} \norm{x}^p_{S_p(C_p)} & \leq \norm{x}^p_{S_p(X_{C_p})} \leq \norm{(x,\theta^n x)}^p_{S_p(X_n \oplus_p C_p)}\\
& = \norm{x}^p_{S_p(X_n)} + \theta^{pn}\norm{x}^p_{S_p(C_p)}
\end{align*}
for any $x \in \K_0 \otimes c_{00}.$

If we choose $n$ large enough so that $\theta^{pn} < \frac{(C+\epsilon)^{-3p}}{2}$, then we get
$$\frac{(C+\epsilon)^{-3p}}{2} \norm{x}^p_{S_p(C_p)} \leq \norm{x}^p_{S_p(X_n)}.$$
Consequently, we have $$d_{cb}(X_n, C_p) \leq \frac{(C+\epsilon)^3}{2^{\frac{1}{p}}}.$$
However, since $C_p \OTm C_p \subseteq CB(R_p, C_p) \cong S_{\frac{2p}{2-p}}$ isometrically
and $S_{\frac{2p}{2-p}}$ is a reflexive Banach space with a basis $C_p \OTm X_n$ does not contain any isomorphic copy of $c_0$,
which is contradictory to Proposition \ref{prop-c0-Cp}.
\end{proof}

\begin{rem}{\rm
Note that if a weak-$H$ space is a homogeneous Hilbertian space, then it is completely isomorphic to $H$ itself by Proposition 3.8. of \cite{L1}.
Thus, $X_{C_p}$ is not homogeneous.
}
\end{rem}

\section{The case $p = 2$}\label{OH}

\subsection{The construction and basic properties of the canonical basis}

We will construct $X_{OH}$, an example of nontrivial weak-$OH$ space, in a similar way.
Many arguments used in section \ref{Cp} still work for $OH$ case also,
so that we only provide the proofs which we need to approach in a significantly different way.

Consider a fixed constant $0<\theta <1$. For $x \in \K_0 \otimes c_{00}$ we define
$$\norm{x}_0 = \norm{x}_{\K \otimes_{\min} (\min\ell_2)}$$
and for $n\geq 0$
$$\norm{x}_{n+1} = \max\Big\{\norm{x}_n, \theta \sup \norm{(E_j x)^{f(k)}_{j=1}}_{\K \otimes_{\min} \ell^{f(k)}_2 (X_n)}\Big\},$$
where the inner supremum runs over all ``allowable" sequence $\{E_j\}^{f(k)}_{j=1} \subseteq \mathbb{N}$.
As before we denote the completion of $(c_{00},\norm{\cdot}_n)$ by $X_n$ and
$X_{n+1}$ inherits the operator space structure from $X_n \oplus_{\infty} \ell_{\infty}(I; \{\ell^{f(k)}_2(X_n)\})$,
where $I$ is the collection of all allowable sequences.

\begin{prop}\label{prop-Hilbertian-OH}
For any $x \in \K_0 \otimes c_{00}$, $(\norm{x}_n)_{n\geq 0}$ is increasing, and we have
$$\norm{x}_{\K \otimes_{\min}(\min \ell_2)} \leq \norm{x}_n \leq \norm{x}_{\K \otimes_{\min} OH}$$
for all $n \geq 0$.
\end{prop}

Now we can consider $\norm{x} = \lim_{n\rightarrow \infty}\norm{x}_n$ for all $x \in \K_0 \otimes c_{00}$,
and $X_{OH}$, the completion of $(c_{00},\norm{\cdot})$ inherits the operator space structure from $\ell_{\infty}(X_n)$.
Moreover, $X_{OH}$ is Hilbertian by Proposition \ref{prop-Hilbertian-OH}.

\begin{prop}\label{prop-another-def}
For any $x \in \K_0 \otimes c_{00}$ and any $n\geq 0$ we have
$$\norm{x}_{n+1} = \max\Big\{\norm{x}_0, \theta \sup \norm{(E_j x)^{f(k)}_{j=1}}_{\K \otimes_{\min} \ell^{f(k)}_2 (X_n)}\Big\},$$
where the inner supremum runs over all ``allowable" sequence $\{E_j\}^{f(k)}_{j=1} \subseteq \mathbb{N}$.
\end{prop}

\begin{prop}
The canonical basis $\{t_i\}_{i\geq 1}$ is a normalized 1-completely unconditional basis for $X_{OH}$.
\end{prop}
\begin{proof}
The only different point from the proof of Proposition \ref{prop-basis-Cp} is the proof for $X_0$. 
Indeed, for any $\sum_{i \geq 1}x_i\otimes t_i \in \K_0 \otimes c_{00}$ we have
\begin{align*}
\norm{\sum_{i \geq 1}a_i x_i\otimes t_i}_0 & = \norm{\sum_{i \geq 1}a_i x_i\otimes e_i}_{\K \OTm \min \ell_2}
= \norm{u : \ell_2 \rightarrow \K, e_i \mapsto a_i x_i} \\
& = \sup\Big\{ \norm{\sum_{i \geq 1}\xi_i a_i x_i} : \sum_{i \geq 1}\abs{\xi_i}^2 \leq 1 \Big\}\\
& \leq \norm{v : \ell_2 \rightarrow \K, e_i \mapsto x_i} \leq \norm{\sum_{i \geq 1} x_i\otimes t_i}_0
\end{align*}
since $\sum_{i \geq 1}\abs{\xi_i a_i}^2 \leq 1$.
\end{proof}

Now we investigate operator space structure spanned by certain disjoint block sequences of $\{t_i\}_{i\geq 1}$. They are
$\theta$-completely isomorphic to operator Hilbert spaces with the same dimensions.

\begin{prop}\label{Prop-blockbasic}
Let $(y_j)^{f(k)}_{j=1}$ be a disjoint and normalized block sequence of $\{t_i\}_{i\geq 1}$ such that
$\text{\rm supp}(y_j)\subseteq \{k, k+1, \cdots\}$ for all $1 \leq j \leq f(k)$. Then we have
$$\theta \Big( \sum^{f(k)}_{j=1} \norm{b_j}^2_{S_2} \Big)^{\frac{1}{2}}
\leq \norm{\sum^{f(k)}_{j=1} b_j \otimes y_j}_{S_2(X_{OH})} \leq \Big( \sum^{f(k)}_{j=1} \norm{b_j}^2_{S_2} \Big)^{\frac{1}{2}}.$$
\end{prop}
\begin{proof}
Consider any ``allowable" sequence $\{E_j\}^{f(k)}_{j=1} \subseteq \mathbb{N}$. Since we have
$$X_{n+1} \rightarrow \ell^{f(k)}_2(X_n),\; x \mapsto (\theta E_j x)^{f(k)}_{j=1}$$ is completely contractive
$$X_{OH} \rightarrow \ell^{f(k)}_2(X_n),\; x \mapsto (\theta E_jx)^{f(k)}_{j=1}$$ is also completely contractive.
Thus if we set $E_j = \text{supp}(y_j)$, then $\{E_j\}^{f(k)}_{j=1}$ is ``allowable", so that we have
\begin{align*}
\norm{\sum^{f(k)}_{j=1}b_j\otimes y_j}_{S_2(X_{OH})}
&\geq\norm{(\theta E_j\Big(\sum^{f(k)}_{i=1}b_i\otimes y_i\Big))^{f(k)}_{j=1}}_{S_2(\ell^{f(k)}_2(X_n))}\\
& = \theta \Big( \sum^{f(k)}_{j=1} \norm{E_j\Big(\sum^{f(k)}_{i=1}b_i \otimes y_i\Big)}^2_{S_2(X_n)} \Big)^{\frac{1}{2}}\\
& = \theta \Big( \sum^{f(k)}_{j=1} \norm{b_j \otimes y_j}^2_{S_2(X_n)} \Big)^{\frac{1}{2}}\\
& = \theta \Big( \sum^{f(k)}_{j=1} \norm{b_j}^2_{S_2}\norm{y_j}^2_{X_n} \Big)^{\frac{1}{2}}
= \theta \Big( \sum^{f(k)}_{j=1} \norm{b_j}^2_{S_2} \Big)^{\frac{1}{2}}.
\end{align*}
The first and the third equality comes from Proposition 2.1 and Lemma 3.6 of \cite{P2}, respectively.

For the right inequality we will show the following more general results using induction on $n$.
\begin{equation}\label{induction-1}
\norm{\sum^{f(k)}_{i=1} b_i \otimes y_i}_{S_2(X_n)} \leq \Big(\sum^{f(k)}_{i=1} \norm{b_i}^2_{S_2} \Big)^{\frac{1}{2}}
\end{equation}
for all $(b_i)^{f(k)}_{i=1} \subseteq \K$ and for any disjoint and normalized sequence $(y_i)^{f(k)}_{i=1}$.

When $n=0$, we have (\ref{induction-1}) since $$CB(OH, \min\ell_2) = B(OH, \min\ell_2)$$ isometrically
and $(e_{j1})_{j\geq 1}$ and $(y_j)^{f(k)}_{j = 1}$ are orthonormal.

Now suppose we have (\ref{induction-1}) for $n$.
By the induction hypothesis and Proposition 2.1 of \cite{P2} we have for any ``allowable" sequence $\{E_j\}^{f(l)}_{j=1} \subseteq \mathbb{N}$ that
\begin{align*}
\theta \norm{(E_j(\sum^{f(k)}_{i=1} b_i \otimes y_i))^{f(l)}_{j=1} }_{S_2(\ell^{f(l)}_2(X_n))}
& =  \Big( \sum^{f(l)}_{j=1} \norm{\sum^{f(k)}_{i=1}  b_i \otimes \theta E_j y_i}^2_{S_2(X_n)} \Big)^{\frac{1}{2}} \\
& \leq \Big[ \sum^{f(l)}_{j=1} \Big(\sum^{f(k)}_{i=1} \norm{b_i}^2_{S_2} \norm{\theta E_j y_i}^2_{X_n}\Big) \Big]^{\frac{1}{2}} \\
& = \Big[ \sum^{f(k)}_{i=1} \norm{b_i}^2_{S_2} \Big( \sum^{f(l)}_{j=1}\norm{ \theta E_j y_i}^2_{X_n}\Big) \Big]^{\frac{1}{2}} \\
& \leq \Big[ \sum^{f(k)}_{i=1} \norm{b_i}^2_{S_2} \norm{y_i}^2_{X_{OH}} \Big]^{\frac{1}{2}}
= \Big[ \sum^{f(k)}_{i=1} \norm{b_i}^2_{S_2}\Big]^{\frac{1}{2}}.
\end{align*}
Thus, we have that $$e_i \mapsto (E_j y_i)^{f(l)}_{j=1}, OH_{f(k)} \rightarrow \ell^{f(l)}_2(X_n)$$ is a complete contraction
for all ``allowable" sequence $\{E_j\}^{f(l)}_{j=1}$, which implies
$$e_i \mapsto y_i, OH_{f(k)} \rightarrow X_{n+1}$$ is also a complete contraction.
Consequently, we get the desired induction result for $n+1$.
\end{proof}

\subsection{$X_{OH}$ is a nontrivial weak-$OH$ space}\label{sec-example-weakH}

\begin{prop}\label{prop-subsp}
For $n\in \mathbb{N}$ we consider $$Y_n = \overline{\text{span}}\{t_i\}_{i\geq n+1} \subseteq X_{OH}.$$
Then for any $E\subseteq Y_n$  with $\text{dim}E = n$, we have $$d_{cb}(E, OH_n) \leq 3\theta^{-1}.$$
\end{prop}

\begin{thm}\label{thm-main1-OH}
$X_{OH}$ is a weak-$OH$ space.
\end{thm}

\begin{lem}\label{lem1}
Let $N \in \mathbb{N}$ be fixed. Then, for any $n\geq 0$ and any $y,z\in X_{OH}$ with $\text{supp}y \subseteq \{1, 2, \cdots, N\}$
and $\text{supp}z \subseteq \{N+1, N+2, \cdots \}$ we have $$\norm{y+z}_{n+1} \leq \max\{ \norm{y}_{n+1} + \alpha \norm{z}_n , \norm{z}_{n+1} \},$$
where $\alpha = \max\{1, \theta \sqrt{f(N)}\}$.
\end{lem}
\begin{proof}
Let's fix $n\geq 0$. If $\norm{y+z}_{n+1} = \norm{y+z}_n$, then it is trivial.

Now we consider any ``allowable" sequence $\{E_j\}^{f(k)}_{j=1} \subseteq \mathbb{N}$.
When $k > N$, we have $$E_j(y+z) = E_jz,$$ so that $$\theta \norm{(E_j(y+z))^{f(k)}_{j=1}}_{\K \OTm \ell^{f(k)}_2(X_n)} \leq \norm{z}_{n+1}.$$
When $k \leq N$, we have
\begin{align*}
\lefteqn{\theta \norm{(E_j(y+z))^{f(k)}_{j=1}}_{\K \OTm \ell^{f(k)}_2(X_n)}}\\
& \leq \theta \norm{(E_jy)^{f(k)}_{j=1}}_{\K \OTm \ell^{f(k)}_2(X_n)} + \theta \norm{(E_jz)^{f(k)}_{j=1}}_{\K \OTm \ell^{f(k)}_2(X_n)} \\
& \leq \norm{y}_{n+1} + \theta \norm{(E_jz)^{f(k)}_{j=1}}_{\ell^{f(k)}_2(\K \OTm X_n)} \\
& = \norm{y}_{n+1} + \theta \sqrt{\sum^{f(k)}_{j=1} \norm{E_jz}^2_n} \\
& \leq \norm{y}_{n+1} + \theta \sqrt{f(N)} \norm{z}_n.
\end{align*}

\end{proof}

\begin{lem}\label{lem2}
For any $n\geq 0$ and any $x\in \K \OTm X_{OH}$ we have $$\norm{x}_{\K \OTm X_{OH}} \leq \norm{(x,\theta^n x)}_{\K \OTm (X_n \oplus_2 OH)}.$$
\end{lem}

\begin{prop}\label{prop-C0copy}
For any $n\geq 0$ and for any infinite dimensional subspace $E \subseteq X_n$ we have that $R \OTm E$ contains an isomorphic copy of $c_0$.
\end{prop}

\begin{thm}\label{thm-main2}
$X_{OH}$ is not completely isomorphic to $OH$.
\end{thm}
\begin{proof}
The only different point from the proof of Theorem \ref{thm-main2-Cp} is to observe that
$R\OTm OH$ does not contain any isomorphic copy of $c_0$. Indeed, we have
\begin{align*}
R\OTm OH & = OH \otimes_h R = [R,C]_{\frac{1}{2}} \otimes_h R\\ & = [R\otimes_h R,C\otimes_h R]_{\frac{1}{2}}.
\end{align*}
Since $R\otimes_h R$ and $C\otimes_h R$ are isometric to $S_2$ and $S_{\infty}$, respectively, $R\OTm OH$ is isometric to $S_4$.
Thus, $R\OTm OH$ is a reflexive Banach space with a basis, so that it does not contain a isomorphic copy of $c_0$.
\end{proof}

\begin{rem}{\rm
Instead of $X_0 = \min \ell_2$ we can use $R+C$ in the above construction for another nontrivial Hilbertian weak-$OH$ space $Y_{OH}$.
However, we do not know $X_{OH}$ and $Y_{OH}$ are completely isomorphic or not at the time of this writing.
}
\end{rem}

\begin{rem}{\rm
By a similar procedure we can construct a non-Hilbertian weak-$H$ space $T_{OH}$. The construction is as follows.

We will define a sequence of norms on $\K_0 \otimes c_{00}$ again. Consider a fixed constant $0<\theta <1$. For $x \in \K_0 \otimes
c_{00}$ we define $$\norm{x}'_0 = \norm{x}_{\K \otimes_{\min} c_0}$$ and for $n\geq 1$ we define
$$\norm{x}'_{n+1} = \max\Big\{\norm{x}'_n, \theta \sup \norm{(E_j x)^{f(k)}_{j=1}}_{\K \otimes_{\min} \ell^{f(k)}_2 (T_n)}\Big\},$$
where the inner supremum runs over all ``allowable" sequence $\{E_j\}^{f(k)}_{j=1} \subseteq \mathbb{N}$.
We denote the completion of $(c_{00},\norm{\cdot}'_n)$ by $T_{n+1}$, and $T_{n+1}$ inherits
the operator space structure from $T_n \oplus_{\infty} \ell_{\infty}(I; \{\ell^{f(k)}_2(T_n)\})$,
where $I$ is the collection of all allowable sequences.

Then by a similar argument as in Proposition \ref{prop-Hilbertian-OH} we can show that
$$\norm{x}_{\K \otimes_{\min}c_0} \leq \norm{x}'_n \leq \norm{x}_{\K \otimes_{\min} OH}$$
for all $n \geq 0$ and $(\norm{x}'_n)_{n\geq 0}$ is increasing.
Thus, $$\norm{x}' = \lim_{n \rightarrow \infty}\norm{x}'_n$$ converges for all $x \in \K_0 \otimes c_{00}$,
so that $T_{OH}$, the completion of $(c_{00},\norm{\cdot}')$ inherits the operator space structure from $\ell_{\infty}(T_n)$.

If we look at the underlying Banach space of $T_{OH}$, then it is nothing but a variant of modified 2-convexification of Tsirelson's space
in (2) of Notes and Remarks in X.e (p.117 of \cite{CS}). The only difference is the fact that we replaced $f(k) = k$ into $f(k) = (4k^3)^k$,
and it is well-known that our $T_{OH}$ is isomorphic to (as a Banach space) the 2-convexified Tsirelson's space
(see X.e. and Appendix b. in \cite{CS}), which is not isomorphic to any Hilbert space.

By a similar argument we can show $T_{OH}$ is a non-Hilbertian example of weak-$OH$ space.

}
\end{rem}

\bibliographystyle{amsplain}
\providecommand{\bysame}{\leavevmode\hbox
to3em{\hrulefill}\thinspace}

\end{document}